\documentclass[10pt]{amsart}
\usepackage{amscd, amsfonts, amsmath, amssymb, amsthm, fixmath, mathrsfs}
\usepackage{hyperref}
\usepackage[all]{xy} 
\makeindex

\theoremstyle{definition} 
\newtheorem{Unity}{Unity}[section] 

\theoremstyle{plain} 
\newtheorem{Theorem}[Unity]{Theorem}
\newtheorem*{Theorem*}{Theorem}
\newtheorem{Proposition}[Unity]{Proposition}
\newtheorem{Corollary}[Unity]{Corollary}
\newtheorem{Lemma}[Unity]{Lemma}

\theoremstyle{remark} 

\newcommand{\K}{\overline{K}}

\newcommand{\N}{\mathbb{N}}

\newcommand{\PP}{\mathbb{P}}

\newcommand{\g}{\mathfrak{g}}
\newcommand{\p}{\mathfrak{p}}

\newcommand{\T}{\mathrm{T}}

\newcommand{\Gal}{\mathrm{Gal}}

\newcommand{\E}{{\mathscr E}}

\newcommand{\Ls}{\mathscr{L}}
\newcommand{\Ox}{\mathscr{O}}

\newcommand{\Omg}{\mathrm{\Omega}}

\newcommand{\Char}{\mathrm{char}}
\newcommand{\Spec}{\mathrm{Spec}}

\newcommand{\Pic}{\mathrm{Pic}}

\newcommand{\rk}{\mathrm{rk}}

\newcommand{\Sym}{\mathrm{Sym}}

\newcommand{\codim}{\mathrm{codim}}

\begin{document}

\title{Semistability of Rational Principal $GL_n$-Bundles in Positive Characteristic}
\author{Lingguang Li}
\address{School of Mathematical Sciences, Tongji University, Shanghai, P. R. China}
\email{LiLg@tongji.edu.cn}
\begin{abstract}
Let $k$ be an algebraically closed field of characteristic $p>0$, $X$ a smooth projective variety over $k$ with a fixed ample divisor $H$. Let $E$ be a rational $GL_n(k)$-bundle on $X$, and $\rho:GL_n(k)\rightarrow GL_m(k)$ a rational $GL_n(k)$-representation at most degree $d$ such that $\rho$ maps the radical $R(GL_n(k))$ of $GL_n(k)$ into the radical $R(GL_m(k))$ of $GL_m(k)$. We show that if $F_X^{N*}(E)$ is semistable for some integer $N\geq\max\limits_{0<r<m}C^r_m\cdot\log_p(dr)$, then the induced rational $GL_m(k)$-bundle $E(GL_m(k))$ is semistable. As an application, if $\dim X=n$, we get a sufficient condition for the semistability of Frobenius direct image ${F_X}_*(\rho_*(\Omega^1_X))$, where $\rho_*(\Omega^1_X)$ is the locally free sheaf obtained from $\Omega^1_X$ via the rational representation $\rho$.
\end{abstract}
\maketitle

\section{Introduction}
Let $k$ be an algebraically closed field of arbitrary characteristic, $X$ a smooth projective variety over $k$ with a fixed ample divisor $H$. Let $G$ and $G'$ be reductive algebraic groups over $k$, $\rho:G\rightarrow G'$ a homomorphism of algebraic groups. One of the important and essential problem in the studying of $G$-bundles is to study the behavior of the semistability of $G'$-bundles under the extension of structure group. In precise, let $E$ be a semistable rational $G$-bundle on $X$, does the induced rational $G'$-bundle $E(G')$ is also semistable?

Suppose that $\rho$ maps the radical of $R(G)$ into the radical $R(G')$ of $G'$ (Unless stated otherwise, we always require this condition for any homomorphisms of algebraic groups and all representations are rational representations in this paper), and $E$ is a semistable rational $G$-bundle on $X$. If char($k$)$=0$, S. Ramanan and A. Ramanathan \cite[Theorem 3.18]{RamananRamanathan84} showed that the induced rational $G'$-bundle $E(G')$ is also semistable on $X$. If char($k$)$=p>0$, the induced rational $G'$-bundle $E(G')$ may be not semistable in general. However, S. Ramanan and A. Ramanathan \cite[Theorem 3.23]{RamananRamanathan84} proved that strong semistability of rational $G$-bundle $E$ implies the strong semistability of rational $G'$-bundle $E(G')$. In addition, S. Ilangovan, V. B. Mehta and A. J. Parameswaran \cite{IlangovanMehtaPara03} showed that if $G'=GL_m(k)$ for some integer $m>0$ and $p>\mathrm{ht}(\rho)$, then the induced rational $GL_m(k)$-bundle $E(GL_m(k))$ is semistable. F. Coiai and Y. I. Holla \cite{CoiaiHolla06} generalized some results of \cite{RamananRamanathan84} and showed that given a representation $\rho:G\rightarrow GL_m(k)$, there exists a non-negative integer $N$, depending only on $G$ and $\rho$, such that for any rational $G$-bundle $E$ whose $N$-th Frobenius pull back $F^{N*}_X(E)$ is semistable, then the induced rational $GL_m(k)$-bundle $E(GL_m(k))$ is again semistable. S. Gurjar and V. Mehta \cite{GurjarMehta11} improved the result of \cite{CoiaiHolla06} and obtain a explicit bound for $N$ in terms of certain numerical data attached to $\rho$. The main ingredient of the proof in \cite{CoiaiHolla06} and \cite{GurjarMehta11} is to give a uniform bound for the field of definition of the instability parabolics (Kempf' parabolic) associated to non-semistable points in related representing space.

We now briefly describe the main idea of their proof. Fix a representation $G\rightarrow GL_m(k)$. Let $E$ be rational $G$-bundle on $X$, $E(G)$ the group scheme over $X$ associated to $E$, and $E(GL_m(k))$ the induced rational $GL_m(k)$-bundle under the extension of structure group via $\rho$.
Let $k(X)$ be the function field of $X$, the generic fiber $E(G)_0$ of $E(G)$ is a group scheme over $\Spec(k(X))$. Let $P$ be a maximal parabolic subgroup of $GL_m(k)$, $E(GL_m(k)/P)$ the associated $GL_m(k)/P$-fiber space over $X$, and $E(GL_m(k)/P)_0$ the generic fiber of $E(GL_m(k)/P)$. Then there is an $E(G)_0$-action on the smooth projective variety $E(GL_m(k)/P)_0$ over $k(X)$ which is linearized by a suitable very ample line bundle. If $E(GL_m(k))$ adimts a reduction of structure group to this maximal parabolic subgroup $P$, then we get a rational section $\sigma:U\rightarrow E(GL_m(k)/P)$, where $U$ is an open subscheme of $X$ with $\codim_X(X-U)\geq 2$. Restricting to the generic fiber gives a $k(X)$-rational point $\sigma_0$ of $E(GL_m(k)/P)_0$. In \cite{RamananRamanathan84}, it is shown that if $\sigma_0$ is a semistable point in $E(GL_m(k)/P)_0$ for the above $E(G)_0$-action, then the rational reduction $\sigma$ does not violate the semistability of rational $GL_m(k)$-bundle $E(GL_m(k))$. Also, if $\sigma_0$ is not semistable for the above $E(G)_0$-action and its instability parabolic $P(\sigma_0)$, which is defined over $\overline{k(X)}$, is actually defined over $k(X)$, then again $\sigma$ does not contradict the semistability of rational $GL_m(k)$-bundle $E(GL_m(k))$. In the case of characteristic $0$, by the uniqueness of instability parabolic and it is invariant under the action of Galois group, then $P(\sigma_0)$ is actually defined over $k(X)$. This proves that $E(GL_m(k))$ is a semistable rational $GL_m(k)$-bundle. However, in the case of characteristic $p>0$, $P(\sigma_0)$ may be not defined over $k(X)$, and it is defined over a finite extension of $k(X)$. By the uniqueness of instability parabolic, the Galois descent argument implies that $P(\sigma_0)$ is actually defined over finite purely inseparable field extension $K^{p{-N}}$ of $K$ for some non-negative integer $N$.

In \cite{CoiaiHolla06} and \cite{GurjarMehta11}, the authors showed that there exists a uniform bound $N$, depending only on $G$ and $\rho$, such that for all possible rational reductions to all maximal parabolic subgroups the instability parabolics of points corresponding to these rational reductions are actually defined over $K^{p{-N}}$ via different methods. This can be shown to imply that if $E$ is a semistable rational $G$-bundle such that $F^{N*}_X(E)$ is semistable, then the induced rational $GL_m(k)$-bundle $E(GL_m(k))$ is also semistable. The major differences between the methods of \cite{CoiaiHolla06} and that of \cite{GurjarMehta11} lie in the approach of estimating the field extension $L$ of $K$ such that a given $K$-scheme $M$ has a $L$-rational point. F. Coiai and Y. I. Holla \cite{CoiaiHolla06} proved the existence of the uniform bound by bounding the non-separability of the group action and the non-reducedness of the stabilizers of various unstable rational points. However, the above estimation does not seen quantifiable. On the other hand, S. Gurjar and V. Mehta \cite{GurjarMehta11} directly estimated the field of definition of the instability parabolics which is probably weaker than the method of \cite{CoiaiHolla06}, but it is quantifiable.

The paper is organized as follows.

In section $2$, we recall some definitions and results about geometric invariant theory and rational principal bundles, such as the instability $1$-PS, instability parabolics of non-semistable points, etc. These results can be found in \cite{Kempf78} and \cite{RamananRamanathan84}.

In section $3$, we mainly study the rationality of the instability parabolics of non-semistable points in $GL_n(k)$-representation spaces, and apply these results to the study of semistability of rational principal bundles under the extension of structure groups via a $GL_n(k)$-representation $\rho:GL_n(k)\rightarrow GL_m(k)$.

\begin{Theorem}[Theorem \ref{Thm:StbExtGroup}]
Let $k$ be an algebraically closed field of characteristic $p>0$, $X$ a smooth projective variety over $k$ with a fixed ample divisor $H$, $\rho:GL_n(k)\rightarrow GL_m(k)$ a $GL_n(k)$-representation over $k$ at most degree $d$. Let $E$ be a rational $GL_n(k)$-bundle on $X$ such that $F_X^{N*}(E)$ is semistable for some integer $$N\geq\max_{0<r<m}C^r_m\cdot\log_p(dr).$$
Then the induced rational $GL_m(k)$-bundle $E(GL_m(k))$ is also semistable.
\end{Theorem}

In section $4$, we study the semistability of truncated symmetric powers $\T^l(\E)$ which is obtained from a torsion free sheaf $\E$ under a special group representations. The main result is to give a sufficient condition for the semistability of $\T^l(\E)$.

\begin{Theorem}[Theorem \ref{Thm:T(V)}]
Let $k$ be an algebraically closed field of characteristic $p>0$, $X$ a smooth projective variety over $k$ with a fixed ample divisor $H$, integer $0\leq l\leq(p-1)\cdot\dim X$. Let $\E$ be a torsion free sheaf of rank $n$ on $X$ such that $F^{N*}_X(\E)$ is semistable for some integer
$$N\geq\max_{1\leq r\leq N(p,n,l)}C^r_{N(p,n,l)}\cdot\log_p(lr),$$
where $N(p,n,l)=\sum\limits^{l(p)}\limits_{q=0}(-1)^q\cdot C^q_n\cdot C^{l-pq}_{n+l-q-1}$, and $l(p)$ is the unique integer such that $0\leq l-l(p)\cdot p< p$. Then the torsion free sheaf $\T^l(\E)$ is also a semistable sheaf.
\end{Theorem}

In section $5$, we study the semistability of the Frobenius direct image of $\rho_*(\Omg^1_X)$ which is obtained from $\Omg^1_X$ via the a $GL_n(k)$-representation $\rho$.

\begin{Theorem}[Theorem \ref{Thm:FroRep}]
Let $k$ be an algebraically closed field of characteristic $p>0$, $X$ a smooth projective variety over $k$ of dimension $n$ with a fixed ample divisor $H$ such that $\deg_H(\Omg^1_X)\geq 0$, $\rho:GL_n(k)\rightarrow GL_m(k)$ a $GL_n(k)$-representation at most degree $d$. Suppose that $F^{N*}_X(\Omg^1_X)$ is semistable for some integer
$$N\geq\max_{1\leq l\leq n(p-1)\atop 0<r<m\cdot N(p,n,l)}C^r_{m\cdot N(p,n,l)}\log_p((d+l)r).$$
Then the Frobenius direct image ${F_X}_*(\rho_*(\Omg^1_X))$ is a semistable sheaf, where $\rho_*(\Omg^1_X)$ is the locally free sheaf obtained from $\Omg^1_X$ via the representation $\rho$. Moreover, in this case, ${F_X}_*(\rho_*(\Omg^1_X)\otimes\Ls)$ is semistable for any line bundle $\Ls\in\Pic(X)$.
\end{Theorem}

\section{Geometric Invariant Theory and Rational Principal Bundles}

\subsection{Geometric Invariant Theory}
Let $K$ be a field, $K_s$ the separable closure of $K$, $\K$ the algebraically closure of $K$. Let $G$ be a connected reductive algebraic group over $K$, Then $G$ has a maximal torus $T$ defined over $K$ (See \cite[Proposition 7.10]{BorelSpringer68}), and $T$ splits over $K_s$. Let $X_*(T)$ be the group of parameter subgroups ($1$-PS) of $T$, i.e., group homomorphisms of the multiplicative group $G_m$ into $T$. Let $N_T$ be the normalizer of $T$ in $G$, then Weyl group $W_T:=N_T/T$ of $G$ with respect to $T$ acts on $X_*(T)$ by conjugation. Fix an inner product $\langle,\rangle$ on $X_*(T)$ which is invariant under the action of Weyl group $W_T$ as well as the Galois group $\Gal(K_s/K)$ (See Section $4$ of \cite{Kempf78}). Then we can define \emph{norm} $\|\lambda\|$ of $1$-PS $\lambda\in X_*(T)$ as $\langle\lambda,\lambda\rangle$. Let $T'$ be another maximal torus of $G$, $T$ is conjugate to $\T'$ by an element of $G$, and the isomorphism $T\rightarrow T'$ is well defined up to Weyl group action on $T$. Therefore the inner product $\langle,\rangle$ on $X_*(T)$ determines uniquely one in $X_*(T')$. Hence the norm of any $1$-PS in $G$ is well defined.

Let $V$ be a finite dimensional $\K$-vector space, $\rho:G(\K)\rightarrow GL_{\K}(V)$ a representation of $G$. A vector $0\neq v\in V$ is \emph{semistable} for the $G(\K)$-action if $0\notin\overline{G(\K)\cdot v}$. One knows that this is equivalent to existence of a $G(\K)$-invariant element $\phi\in S^m(V)$ for some $m>0$ such that $\phi(v)\neq0$. For a $1$-PS $\lambda$ of $G(\K)$, $V$ has a decomposition $V=\bigoplus V_i$, where $V_i=\{v\in V|\lambda(t)v=t^i v\}$. Define $$m(\lambda,v):=\min\{i|v~\text{has a non-zero component in}~V_i\},$$ $$\mu(\lambda,v):=\frac{m(\lambda,v)}{\|\lambda\|}.$$

For a $1$-PS $\lambda$ of $G(\K)$, the associated subgroup $P(\lambda)$ of $G(\K)$ is defined by $$P(\lambda):\{g\in G(\K)|\lim\limits_{t\rightarrow 0}\lambda(t)\cdot g\cdot \lambda^{-1}(t)~\text{exists in}~G(\K)\},$$ which is a parabolic subgroup of $G(\K)$.

For a non-semistable vector $v\in V$, define the \emph{instability} $1$-PS of $v$ to be the $1$-PS $\lambda_v$ such that $\mu(\lambda_v,v)=\sup\{\mu(\lambda,v)|\lambda\in X_*(G(\K))\}$, which is not unique.

We now recall some basic facts in geometric invariant theory, which can be found in \cite{Kempf78} and \cite{RamananRamanathan84}.

\begin{Lemma}\cite{Kempf78}\cite{RamananRamanathan84}\label{Lem:InstbParabolic}
Let $K$ be a field, $G$ a connected reductive algebraic group over $K$, $\rho:G(\K)\rightarrow GL_{\K}(V)$ a representation of $G(\K)$ on a $\K$-vector space $V$ of finite dimension. Let $v\in V$ be a non-semistable point for the $G(\K)$-action. Then
\begin{itemize}
    \item[$(1)$] The function $\lambda\mapsto\mu(\lambda,v)$ on the set $X_*(G(\K))$ attains the maximum value.
    \item[$(2)$] There is a unique instability parabolic $P(v)$ of $v$ such that for any instability $1$-PS $\lambda$ of $v$, we have $P(v)=P(\lambda)$.
    \item[$(3)$] For any maximal torus $T\subseteq P(v)$, there is a unique instability $1$-PS $\lambda_T$ of $v$ such that $\lambda_T\subseteq T$.
    \item[$(4)$] For $g\in G(\K)$, $\lambda$ is the instability $1$-PS of $v$, then $g\cdot\lambda(t)\cdot g^{-1}$ is the instability $1$-PS of $g\cdot v$, $\mu(\lambda,v)=\mu(g\lambda(t)g^{-1},g\cdot v)$ and $P(g\cdot v)=g\cdot P(v)\cdot g^{-1}$.
    \item[$(5)$] For an instability $1$-PS $\lambda$ of $v$, if $\lambda$ is defined over an extension field $L/K$, then the instability parabolic $P(v)$ of $v$ is also defined over $L/K$.
\end{itemize}
\end{Lemma}

\subsection{Rational Principal $G$-Bundles}
Let $K$ be a field, $G$ a reductive algebraic group over $K$, $X$ a smooth projective variety over $K$ with a fixed ample divisor $H$. A (\emph{principal}) \emph{$G$-bundle} on $X$ is a $X$-scheme $\pi:E\rightarrow X$ with an $G$-action (acts on the right) and $\pi$ is $G$-invariant and isotrivial, i.e., locally trivial in the \'{e}tale topology.

If $Y$ is a quasi projective $G$-scheme over $K$ (on the left), the associated fibre bundle $E(Y)$ over $X$ is the quotient $E\times_KF$ under the action of $G$ given by $$g(e,y)=(e\cdot g,g^{-1}\cdot y), g\in G, e\in E, y\in Y.$$

Let $M$ be a projective variety over $K$ with a $G$-action, which is linearized by an ample line bundle $\Ls$ on $X$. Let $E(G):=E\times_{G,\text{Int}}G$ denote the group scheme over $X$ associated to $E$ by the action of $G$ on itself by inner automorphisms. Then $X$-group scheme $E(G)$ acts naturally on the $X$-scheme $E(M)$ which is linearized by line bundle $E(\Ls)$.

Let $x\in X$ be a point of $X$, $E(G)_x$, $E(M)_x$ and $E(\Ls)_x$ denote the fiber of $E(G)$, $E(M)$ and $E(\Ls)$ over $x$ respectively.
Then $E(G)_x$ is a group scheme over $\Spec(k(x))$, and one has the action of $E(G)_x$ on $E(M)_x$ linearized by line bundle $E(\Ls)_x$ which is defined over $\Spec(k(x))$.

Let $P\subset G$ be a closed subgroup of $G$, a \emph{reduction of structure group} of $E$ to $P$ is a pair $\sigma:=(E_{\sigma},\phi)$ with a $P$-bundle $E_{\sigma}$ and an isomorphism of $G$-bundles $\phi:E_{\sigma}(G)\rightarrow E$. Note that quotient $E/P$ is naturally isomorphic to the fiber bundle $E(G/P)$ on $X$ and a section $\sigma:X\rightarrow E/P$ gives the $P$-bundle $\sigma^*(E)$ with natural isomorphism $\sigma^*(E)(G)\cong E$. This induces a bijection correspondence between sections of $E/P\rightarrow X$ and reductions of structure group of $E$ to $P$. Let $T_P$ be the tangent bundle along the fibers of the map $E/P\rightarrow X$, then $T_{\sigma}:=\sigma^*(T_P)$ is the vector bundle on $X$ associated to $P$-bundle $E_{\sigma}$ for the natural representation of $P$ on $\g/\p$, where $\g$ and $\p$ is the Lie algebra of $G$ and $P$ respectively.

A \emph{rational $G$-bundle} $E$ on $X$ is a $G$-bundle on a big open subscheme $U$ of $X$, i.e. $\codim_X(X-U)\geq 2$. A \emph{rational reduction of structure group} of a rational $G$-bundle $E$ to a subgroup $P\subset G$ is a reduction $\sigma$ of structure group of $E|_{U'}$ to $P$ over a big open subscheme $U'\subseteq U$. Then the locally free sheaf $T_{\sigma}$ on $U'$ determine a reflexive sheaf ${i_{U'}}_*(T_{\sigma})$ where $i_{U'}:U'\rightarrow X$ is the natural open immersion, denoted by $T_{\sigma}$ again. The rational $G$-bundle $E$ is \emph{semistable} if for any rational reduction $\sigma$ of $E$ to any parabolic subgroup $P$ of $G$ over any big open subscheme $U'\subseteq U$, the rational vector bundle $T_{\sigma}$ has $\deg_H(T_{\sigma})\geq 0$. If $G=GL_n(K)$, then there is an one to one correspondence between rational $GL_n(K)$-bundles between reflexive torsion free sheaves. In this case, the semistability of the rational $GL_n(K)$-bundle $E$ is equivalent to the semistability of the torsion free sheaf $\E$ in the sense of Mumford-Takemoto, where $\E$ is the reflexive torsion free sheaf corresponds to $E$.

\subsection{Frobenius Pull Back of Principal $G$-Bundles}
Let $K$ be a field of characteristic $p>0$, $\phi:X\rightarrow\Spec(K)$ a scheme over $k$. The \emph{absolute Frobenius morphism} $F_X:X\rightarrow X$ is induced by $\Ox_X\rightarrow\Ox_X$, $f\mapsto f^p$. Consider the commutative diagram
$$\xymatrix{
  X \ar@/_/[ddr]_{\phi} \ar@/^/[drr]^{F_X} \ar@{>}[dr]|-{F_g}\\
   & X^{(1)} \ar[d]^{F_k^*(\phi)} \ar[r]^{F_a} & X \ar[d]^{\phi}\\
   & \Spec(K) \ar[r]^{F_K}             & \Spec(K).}$$
The morphism $F_a$ (resp. $F_g$) is called the \emph{arithmetic Frobenius morphism} (resp. \emph{geometric Frobenius morphism}) of $\phi:X\rightarrow\Spec(K)$.

If $K$ is a perfect field, $F_K$ and $F_a$ are isomorphisms. Let $G$ be an algebraic group over $K$, $\pi:E\rightarrow X$ a $G$-bundle over $X$. Pulling back by the absolute Frobenius $F^*_X$ we get a $G$-bundle $F^*_X(\pi):F^*_X(E)\rightarrow X_F$, where $X_F$ is the scheme $X$ endowed with the $K$-structure by the composition $X\stackrel{\phi}{\rightarrow}\Spec(K)\stackrel{F_K}{\rightarrow}\Spec(K)$. If $K$ is a perfect field, we can change the $K$-structures of $G$, $X_F$ and $F^*_X(E)$ by composing their structure morphisms with $\Spec(K)\stackrel{F^{-1}_K}{\rightarrow}\Spec(K)$ to get a $F^*_K(G)$-bundle $F^*_X(E)\rightarrow X$. In this case, the $F^*_K(G)$-bundle $F^*_X(E)\rightarrow X$ is the same as the bundle obtained from $G$-bundle $E$ by the extension of structure group $g:G\rightarrow F^*_K(G)$.

Let $k$ be an algebraically closed field of characteristic $p>0$, $X$ a smooth projective variety over $k$ with a fixed ample line bundle $H$, $G$ a reductive algebraic group over $k$. Let $E$ be a rational $G$-bundle on $X$, then we can get a rational $F^*_K(G)$-bundle $F^*_X(E)$ on $X$. Then $E$ is semistable when $F^*_X(E)$ is semistable. If $F^{m*}_X(E)$ is semistable for any integer $m\geq 0$, then $E$ is called \emph{strongly semistable}.

\subsection{Semistability of Rational $G$-bundles under Extension of Structure Groups}
Let $k$ be an algebraically closed field, $X$ a smooth projective variety over $k$ with a fixed ample line bundle $H$, $G$ a reductive algebraic group over $k$ with a rational representation $\rho:G\rightarrow GL_m(k)$. Let $E$ be a semistable rational $G$-bundle on $X$, we study the semistability of rational $GL_m(k)$-bundle $E(GL_m(k))$.

Let $P$ be a maximal parabolic subgroup of $GL_m(k)$, then $G$ acts on $GL_m(k)/P$ which is linearized by the very ample generator $\Ls$ of $\Pic(GL_m(k)/P)$. This gives an $G$-invariant embedding of $GL_m(k)/P$ inside projective space $\PP(H^0(X,\Ls))$. Then $X$-group scheme $E(G)$ acts naturally on the $X$-scheme $E(GL_m(k)/P)$ which is linearized by line bundle $E(\Ls)$.

Let $E(G)_0$, $E(GL_m(k)/P)_0$ and $E(\Ls)_0$ be the fiber of $E(G)$, $E(GL_m(k)/P)$ and $E(\Ls)$ over the generic point of $X$ respectively. Then $E(G)_0$ is a group scheme over function field $\Spec(k(X))$, and one has the action of $E(G)_0$ on $E(GL_m(k)/P)_0$ linearized by line bundle $E(\Ls)_0$ which is defined over $\Spec(k(X))$. Therefore there is an one to one correspondence between rational reductions of structure group of rational $GL_m(k)$-bundle $E(GL_m(k))$ to $P$ and $k(X)$-rational points of $E(GL_m(k)/P)_0$.

\begin{Lemma}\cite[Proposition 3.10, Proposition 3.13]{RamananRamanathan84}\label{InstbParabolic}
Let $k$ be an algebraically closed field, $X$ a smooth projective variety over $k$ with an ample line bundle $H$. Let $G$ be a reductive algebraic group over $k$, $E$ a semistable rational $G$-bundle on $X$. Let $\sigma:U'\rightarrow E(GL_m(k)/P)$ be a rational reduction of structure group of $GL_m(k)$-bundle $E(GL_m(k))$ to a maximal parabolic subgroup $P$ of $GL_m(k)$, where $U'$ is a big open scheme of $X$. The associated $k(X)$-rational point in $E(GL_m(k)/P)_0$ is denoted by $\sigma_0$. Then
\begin{itemize}
    \item[$(1)$] If $\sigma_0$ is a semistable point for the action of $E(G)_0$ on $E(GL_m(k)/P)_0$ linearized by line bundle $E(\Ls)_0$, then $\deg_HT_{\sigma}\geq 0$.
    \item[$(2)$] If $\sigma_0$ is not a semistable point for the action of $E(G)_0$ on $E(GL_m(k)/P)_0$ linearized by line bundle $E(\Ls)_0$, and its instability parabolic $P(\sigma)$ is defined over $k(X)$, then $\deg_HT_{\sigma}\geq 0$.
\end{itemize}
\end{Lemma}

\section{The Semistability of Principal Bundles via $GL_n$-Representaions}

Let $K$ be an arbitrary field, $\rho:GL_n(K)\rightarrow GL_m(K)$ a representation of $GL_n(K)$ over $K$. Then $\rho$ is given by an $m\times m$-matrix of regular functions $$\frac{f_{ij}}{\det(T_{ij})^{a_{ij}}}\in K[T_{ij},\det(T_{ij})^{-1}]_{1\leq i,j\leq n},$$ where $a_{ij}\in\N$ and $f_{ij}\in K[T_{ij}]_{1\leq i,j\leq n}$ with $(\det(T_{ij}),f_{ij})=1$. Denote $$d:=\max\limits_{1\leq i,j\leq n}\{\deg(f_{ij}+n(a-a_{ij})\},~\text{where}~a:=\max\limits_{1\leq i,j\leq n}\{a_{ij}\}.$$
We say $\rho$ is a \emph{$GL_n(K)$-representation over $K$ of dimension $m$ at most degree $d$.}
Moreover, if $a_{ij}=0$ for any $1\leq i,j\leq n$, i.e., the above regular functions on $GL_n(K)$ lie in the subring $K[T_{ij}]_{1\leq i,j\leq n}$, we say that $\rho$ is a \emph{polynomial representation} of $GL_n(K)$ over $K$.

In this section, we use a variant of method of S. Gurjar and V. Mehta \cite{GurjarMehta11} to give an explicit uniform bound for the field of definition of all instability parabolics of all non-semistable points in a given $GL_n(k)$-representation space $V$ in terms of $n$, $\dim_kV$ and the maximal degree of regular functions correspond to the given $GL_n(k)$-representation. Moreover, we use these results to the study of semistability of principal bundles under the extension of structure groups via $GL_n(k)$-representations.

\begin{Theorem}\label{Thm:PolyRep}
Let $K$ be an arbitrary field of characteristic $p>0$, $V$ a $K$-vector space of dimension $m$. Let $\rho:GL_n(K)\rightarrow GL_K(V)$ be a $GL_n(K)$-representation over $K$ at most degree $d$. Then the instability parabolic of any unstable $K$-rational point in $V$ is defined over $K^{p^{-N}}$ for some positive integer $N\geq m\cdot\log_p(d)$.
\end{Theorem}
\begin{proof}
By base change, the representation $\rho:GL_n(K)\rightarrow GL_K(V)$ over $K$ induces a representation $\overline{\rho}:GL_n(\K)\rightarrow GL_{\K}(V_{\K})$ over $\K$ at most degree $d$. Fix a maximal torus $T$ in $GL_n(\K)$ which is defined over $K$. Denote $V_{\K}:=V\otimes_K\K$, we can choose a simultaneous eigen basis $e_1,\ldots,e_m$ of $V_{\K}$ for all $1$-PS of $GL_n(k)$ which lie in $T$.

Let $v\in V$ be a non-semistable $K$-rational point with respect to the $GL_n(\K)$-action $\overline{\rho}$. Then there are $f_l\in K[T_{ij}]_{1\leq i,j\leq n}$ of $\deg(f_l)\leq d$ and $a\in\N$, such that for any $\K$-rational point $g\in GL_n(\K)$, we have $$g\cdot v=\sum^m_{l=1}\frac{f_l(g)}{\det(g)^a}e_l.$$
Let $\lambda(t)\in X_*(GL_n(\K))$ be an instability $1$-PS of $v$. Then there exists $\K$-rational point $g\in GL_n(\K)$ such that $g\cdot\lambda(t)\cdot g^{-1}\subseteq T$. Then $g\cdot\lambda(t)\cdot g^{-1}$ is defined over $K_s$, the separable closure of $K$, and $g\cdot\lambda(t)\cdot g^{-1}$ is an instability $1$-PS of $g\cdot v$ with $\mu(\lambda(t),v)=\mu(g\cdot\lambda(t)\cdot g^{-1},g\cdot v)$.

Let $f_{l_1},\ldots, f_{l_r}$ (resp. $f_{l_{r+1}},\ldots, f_{l_m}$) denote the set of polynomials which vanish at $g$ (resp. non-vanish at $g$). Consider the $K$-affine scheme $$X:=\Spec(K[T_{ij}]_{1\leq i,j\leq n}/(f_{l_1},\ldots, f_{l_r})).$$
Then $g$ is a $\K$-rational point of $X(\K)\subseteq\mathbb{A}^{n\times n}(\K)$ with $$\det(g)\prod^m_{l=r+1}f_l(g)\neq 0.$$
Therefore, by \cite[Lemma 10]{GurjarMehta11}, there exists a finite extension field $L\subseteq\K$ of $K$ with $$[L:K]\leq\prod^r_{l=1}\deg(f_l)=d^r$$ such that $X$ has a $L$-rational point $g'$ in $X$ and $\det(g')\prod^m_{l=r+1}f_l(g')\neq0$. Thus $g'\in GL_n(L)$ and $g\cdot v$ and $g'\cdot v$ have the same set of monomials with non-zero coefficients when expanded in terms of the basis $e_1,\ldots,e_m$. Since $e_1,\ldots,e_m$ is a simultaneous basis for $g\cdot\lambda(t)\cdot g^{-1}$, so $\mu(g\cdot\lambda(t)\cdot g^{-1},g\cdot v)=\mu(g\cdot\lambda(t)\cdot g^{-1},g'\cdot v)$. Hence
$$\mu(g'\cdot\lambda(t)\cdot g'^{-1},g'\cdot v)=\mu(\lambda(t),v)=\mu(g\cdot\lambda(t)\cdot g^{-1},g\cdot v)=\mu(g\cdot\lambda(t)\cdot g^{-1},g'\cdot v).$$
As $g'\cdot\lambda(t)\cdot g'^{-1}$ is an instability $1$-PS of $g'\cdot v$, so $g\cdot\lambda(t)\cdot g^{-1}$ is also an instability $1$-PS of $g'\cdot v$. It follows that $g'^{-1}(g\cdot\lambda(t)\cdot g^{-1})g'$ is an instability $1$-PS of $v$ which is defined over $L$. Thus, by Lemma \ref{Lem:InstbParabolic}, the instability parabolic $P(v)$ of $v$ is defined over $L\cdot K_s$. By the Galois descent argument, any instability parabolic is defined over a purely inseparable extension of $K$. Suppose that for some positive integer $N$ with $p^N\geq d^m\geq d^r$. Then $P(v)$ must be defined over $L\cap K^{p^{-N}}$. By the arbitrariness of non-semistable point $v\in V$, this theorem is followed.
\end{proof}

\begin{Theorem}\label{Thm:StbExtGroup}
Let $k$ be an algebraically closed field of characteristic $p>0$, $X$ a smooth projective variety over $k$ with a fixed ample divisor $H$, $\rho:GL_n(k)\rightarrow GL_m(k)$ a $GL_n(k)$-representation over $k$ at most degree $d$. Let $E$ be a rational $GL_n(k)$-bundle on $X$ such that $F_X^{N*}(E)$ is semistable for some integer $$N\geq\max_{0<r<m}C^r_m\cdot\log_p(dr).$$
Then the induced $GL_m(k)$-bundle $E(GL_m(k))$ is also semistable.
\end{Theorem}
\begin{proof}
Let $V$ be a $k$-vector space of dimension $m$ with a fixed basis $e_1,\ldots,e_n$. Then there is a natural isomorphism $GL_n(k)\cong GL_k(V)$ with respect to this basis of $V$. Let $P$ be a maximal parabolic subgroup of $GL_k(V)$, then $GL_k(V)/P$ is isomorphic to the grassmannian $Grass(r,V)$ of $r$-dimensional subspaces of $V$ for some integer $0<r<\dim_kV$. This induces a $GL_n(k)$-equivalent embedding $$ GL_k(V)/P\hookrightarrow\PP(\bigwedge^rV),$$ which can be lift to a $GL_n(k)$-representation $$\rho^r:GL_n(k)\rightarrow GL_k(\bigwedge^rV).$$ Since $\rho:GL_n(k)\rightarrow GL_k(V)$ is a $GL_n(k)$-representation over $k$ at most degree $d$, $\rho^r$ is a $GL_n(k)$-representation over $k$ at most degree $dr$. (In fact, $GL_n(k)$ acts on $GL_k(V)/P$ which is linearized by the very ample generator $\Ls$ of $\Pic(GL_k(V)/P)$. This gives a $GL_n(k)$-invariant embedding of $GL_k(V)/P$ inside projective space $\PP(H^0(X,\Ls))$, and a $GL_n(k)$-equivalent isomorphism $\bigwedge^rV\cong H^0(X,\Ls)$.)

The $X$-group scheme $E(GL_n(k))$ acts naturally on the $X$-scheme $E(GL_k(V)/P)$ which is linearized by line bundle $E(\Ls)$. Let $E(G)_0$, $E(GL_k(V)/P)_0$ and $E(\Ls)_0$ denote the fiber of $E(G)$, $E(GL_k(V)/P)$ and $E(\Ls)$ over the generic point of $X$ respectively. Then $E(G)_0$ is a group scheme over function field $\Spec(k(X))$, and one has an $E(G)_0$-action on $E(GL_k(V)/P)_0$ defined over $\Spec(k(X))$,  which is linearized by line bundle $E(\Ls)_0$. One observes that $E(GL_n(k))_0$ and $E(GL_k(V)/P)_0$ became trivial after a finite separable field extension. Hence after change the base to $k(X)_s$, the separable closure of $k(X)$, we get canonical isomorphisms
$$E(GL_n(k))_0\otimes_{k(X)}k(X)_s\cong GL_n(k)\otimes_{k}k(X)_s\cong GL_n(k(X)_s),$$
$$E(GL_k(V)/P)_0\otimes_{k(X)}k(X)_s\cong(GL_k(V)/P)\otimes_{k}k(X)_s,$$
$$E(\Ls)_0\otimes_{k(X)}k(X)_s\cong\Ls\otimes_{k}k(X)_s,$$
and this isomorphisms being compatible with group actions. Let $V_s:=V\times_kk(X)_s$. The $E(GL_n(k))_0\otimes_{k(X)}k(X)_s$-action on $E(GL_k(V)/P)_0\otimes_{k(X)}k(X)_s$ induces a $GL_n(k(X)_s)$-equivalent embedding $$(GL_m(k)/P)\otimes_{k}k(X)_s\hookrightarrow\PP(\bigwedge^rV)\otimes_{k}k(X)_s\cong\PP(\bigwedge^rV_s),$$
which can be lift to a $GL_n(k(X)_s)$-representation $$\rho^r_s:GL_n(k(X)_s)\rightarrow GL_{k(X)_s}(\bigwedge^rV_s).$$
In fact, $\rho^r_s$ is just the base change of the $GL_n(k)$-representation $\rho^r$ from $k$ to $k(X)_s$. Hence $\rho^r_s$ is a $GL_n(k(X)_s)$-representation over $k(X)_s$ of dimension $C^r_m$ at most degree $dr$.

Let $\sigma:U\rightarrow E(GL_k(V)/P)$ be a rational reduction of structure group of $E$ to $P$, where $U$ is an open subscheme of $X$ with $\codim_X(X-U)\geq2$. It corresponds to a $k(X)$-rational point $\sigma_0\in E(GL_k(V)/P)_0(k(X))$. Let $T_\sigma$ denote the torsion free sheaf determined by the the locally free sheaf $\sigma^*(T_P)$ on $U$, where $T_P$ is the tangent bundle along the fibers of the map $E(GL_k(V)/P)\rightarrow X$. By Lemma \ref{InstbParabolic}, if $\sigma_0$ is a semistable point for the action of $E(GL_n(k))_0$ on $E(GL_k(V)/P)_0$ linearized by line bundle $E(\Ls)_0$, then $$\deg_HT_{\sigma}\geq 0.$$ On the other hand, if $\sigma_0$ is not a semistable point for the action of $E(GL_n(k))_0$ on $E(GL_k(V)/P)_0$ linearized by line bundle $E(\Ls)_0$, we would like to prove that $$\deg_HT_{\sigma}\geq 0.$$

View $\sigma_0$ as an $k(X)_s$-rational point in $(GL_k(V)/P)\otimes_{k}k(X)_s$, and lift this point to a point in $\bigwedge^rV_s$, denote by $\sigma_0$ again. Then by Theorem \ref{Thm:PolyRep}, we have the instability parabolic $P(\sigma_0)$ of $\sigma_0$ is defined over $k(X)^{p^{-l}}_s$ for any integer $l\geq C^r_m\cdot\log_p(dr)$. Then $P(\sigma_0)$ is actually defined over $k(X)^{p^{-l}}$ by the uniqueness of the instability parabolic and Galois descent argument.

Pulling back by the Frobenius morphism, the action of the generic fibre $$F^{l*}_X(E(GL_n(k)))_0\cong F^{l*}_{k(X)}(E(GL_n(k))_0)$$ on $$F^{l*}_X(E(GL_k(V)/P))_0\cong F^{l*}_{k(X)}(E(GL_m(V)/P)_0)$$ is the base change by $F^{l*}_{k(X)}$ of the $E(GL_n(k))_0$-action on $(GL_k(V)/P)_0$. The Frobenius $F^{l*}_{k(X)}$ factors through an isomorphism:
$$\xymatrix{
   & k(X) \ar[d]_{\cong} \ar[r]^{F^{l*}_{k(X)}} & k(X)\\
   & k(X)^{p^{-l}} \ar[ur]_{i_l}}$$
where $i_l:\Spec(k(X)^{p^{-l}})\rightarrow\Spec(k(X))$ is given by the inclusion $k(X)\subseteq k(X)^{p^{-l}}$.
Therefore for this $F^{l*}_{k(X)}(E(GL_n(k))_0)$-action, the instability parabolic of the point $F^{l*}_X(\sigma)_0$ is defined over $k(X)$.
Since $F^{l*}_X(E)$ is semistable, by Lemma \ref{InstbParabolic}, we have $$\deg_H(T_{F^{l*}_X(\sigma)})=\deg_H F^{l*}_X(T_{\sigma})=p^l\cdot\deg_H(T_{\sigma})\geq 0.$$
Thus $\deg_HT_{\sigma}\geq 0$. Therefore if $F^{N*}_X(E)$ is semistable for some integer $$N\geq\max_{0<r<m}C^r_m\cdot\log_p(dr),$$ then for any rational reduction $\sigma$ of structure group of rational $GL_m(k)$-bundle $E(GL_m(k))$ to any maximal parabolic subgroup $P\subset GL_m(k)$, we have $$\deg_H(T_{\sigma})\geq 0.$$ Hence $E(GL_m(k))$ is a semistable rational $GL_m(k)$-bundle. This completes the proof of this theorem.
\end{proof}

\section{The Semistability of Truncated Symmetric Powers}

The truncated symmetric powers was first introduced in \cite{Sun08} in order to study the semistability of Frobenius direct images. L. Li and F. Yu \cite{LiYu13} have studied the instability of $\T^l(\E)$ and show that $\T^l(\E)$ is strongly semistable when $\E$ is strongly semistable. In this section, we would like to continue the further study of the semistability $\T^l(\E)$, and give a sufficient condition for semistability of $\T^l(\E)$.

Now, we recall the construction and properties of truncated symmetric powers of vector spaces (See \cite[Section 3]{Sun08}).

Let $K$ be an arbitrary field, $V$ a $n$-dimensional $K$-vector space with standard representation of $GL_n(K)$. Let $l$ be a positive integer, $S_l$ the symmetric group of $l$ elements with a natural action on $V^{\otimes l}$ by $$(v_1\otimes\cdots\otimes v_l)\cdot\sigma=v_{\sigma_{(1)}}\otimes\cdots\otimes v_{\sigma_{(l)}}$$ for any $v_i\in V$ and any $\sigma\in S_l$.
Let $e_1,\cdots,e_n$ be a basis of $V$. For any non-negative partition $(k_1,\cdots,k_n)$ of $l$ (i.e. $l=\sum\limits^n_{i=1}k_i$, $k_i\geq 0,~1\leq i\leq n$), we define $$v(k_1,\cdots,k_n):=\sum\limits_{\sigma\in S_l}(e^{\otimes k_1}_1\otimes\cdots\otimes e^{\otimes k_n}_n)\cdot\sigma.$$
Let $\T^l(V)\subset V^{\otimes l}$ be the linear subspace generated by all vectors $$\{v(k_1,\cdots,k_n)~|~l=\sum\limits^n_{i=1}k_i,~k_i\geq 0,~1\leq i\leq n\}.$$
Then $\T^l(V)$ is a $GL_n(K)$-subrepresentation of $V^{\otimes l}$ with $$N(p,n,l):=\dim_k\T^l(V)=\sum^{l(p)}_{q=0}(-1)^q\cdot C^q_n\cdot C^{l-pq}_{n+l-q-1},$$
where $l(p)$ is the unique integer such that $0\leq l-l(p)\cdot p< p$.

If $\Char(K)=0$ then we have $GL_n(K)$-equivalent $\T^l(V)\cong\Sym^l(V)$ for any integer $l>0$. On the other hand, if $\Char(K)=p>0$, then we have $GL_n(K)$-equivalent $\T^l(V)\cong\Sym^l(V)$ when $0<l<p$ and $\T^l(V)=0$ for $l>n(p-1)$.

\begin{Proposition}
For any integer $l>0$, $\T^l(V)=(V^{\otimes l})^{S_l}$.
\end{Proposition}
\begin{proof}
Fix a basis $e_1,\cdots,e_n$ of $V$. Let $(k_1,\cdots,k_n)$ be a non-negative partition of $l$, $W_{k_1,\cdots,k_n}$ the linear subspace of $V^{\otimes l}$ generated by vectors $$\{e_{i_1}\otimes\cdots\otimes e_{i_l}|~k_m=\sharp\{i_j|i_j=m,1\leq j\leq l\}, 1\leq m\leq n\}.$$
Thus $W_{(k_1,\cdots,k_n)}$ is a $S_l$-invariant linear subspace of $V^{\otimes l}$.

By definition, it is obvious that $\T^l(V)\subseteq(V^{\otimes l})^{S_l}$. It is easy to see that, any element in $(V^{\otimes l})^{S_l}$ can be expressed as the form $$\alpha=\sum_{l=\sum\limits^n_{i=1}k_i,k_i\geq 0}\alpha_{k_1,\cdots,k_n},$$ where $\alpha_{k_1,\cdots,k_n}\in W_{k_1,\cdots,k_n}$. Then we have $\alpha_{k_1,\cdots,k_n}\in(V^{\otimes l})^{S_l}$. In order to prove $\alpha\in\T^l(V)$, it suffices to show that $\alpha_{k_1,\cdots,k_n}\in\T^l(V)$.
By simple observation, we have $\alpha_{k_1,\cdots,k_n}=a\cdot v(k_1,\cdots,k_n)$ for some $a\in k$. It follows that $(V^{\otimes l})^{S_l}\subseteq\T^l(V)$. Hence $\T^l(V)=(V^{\otimes l})^{S_l}$.
\end{proof}

\begin{Proposition}\label{Prop:l}
Let $K$ be a field of characteristic $p>0$, $V$ a $n$-dimensional $K$-vector space. Then for any integer $0<l\leq n(p-1)$, the $GL_K(V)$-representation $\rho^l:GL_K(V)\rightarrow GL_K(\T^l(V))$ is a polynomial representation at most degree $l$.
\end{Proposition}
\begin{proof}
Fix a basis $e_1,\cdots,e_n$ of $V$. Endowing the basis $\{e_{i_1}\otimes\cdots\otimes e_{i_l}~|~1\leq i_j\leq n\}$ of $V^{\otimes l}$ with lexicographic order. Under this basis, the $GL_K(V)$-representation $V^{\otimes l}$ is equivalent to the homomorphism of algebraic groups $$\rho^l:GL_n(K)\rightarrow GL_{n^l}(K)$$
$$(g_{ij})\mapsto (h_{st}):=((g_{ij})\otimes\cdots\otimes(g_{ij})).$$
Therefore $\rho^l$ is given by a matrix of polynomials $h_{st}$ in $K[T_{ij}]_{1\leq i,j\leq n}$ of degree $l$ for any integers $1\leq s,t\leq n^l$. Then by the change of basis, under the basis $e_1,\cdots,e_n$, $$(\text{resp.}~\{v(k_1,\cdots,k_n)~|~l=\sum\limits^n_{i=1}k_i,~0\leq k_i<p,~1\leq i\leq n\})$$ of $V$ (resp. $\T^l(V)$), the $GL_K(V)$-representation $\rho^l:GL_K(V)\rightarrow GL_K(\T^l(V))$ is also give by a matrix of polynomials $f_{ij}$ in $K[T_{ij}]_{1\leq i,j\leq n}$ of degree $l$ for any integers $1\leq i,j\leq\dim\T^l(V)$. Hence $\rho^l:GL_K(V)\rightarrow GL_K(\T^l(V))$ is a polynomial representation at most degree $l$.
\end{proof}

\begin{Corollary}
Let $K$ be a field of characteristic $p>0$, $V$ a $n$-dimensional $K$-vector space, $l$ an integer with $0<l\leq n(p-1)$. Then the instability parabolic of any unstable $K$-rational point in the $GL_K(V)$-representation space $\T^l(V)$ is defined over $K^{p^{-N}}$ for some integer $N\geq\log_p(l)\cdot\dim_K\T^l(V)$, i.e., $$N\geq\log_p(l)\cdot\sum^{l(p)}_{q=0}(-1)^q\cdot C^q_n\cdot C^{l-pq}_{n+l-q-1},$$
where $l(p)$ is the unique integer such that $0\leq l-l(p)\cdot p< p$.
\end{Corollary}
\begin{proof}
By Proposition \ref{Prop:l}, we have $\rho^l:GL_K(V)\rightarrow GL_K(\T^l(V))$ a polynomial representation of $GL_K(V)$ at most degree $l$. Hence by Theorem \ref{Thm:PolyRep}, if an integer $N\geq(\dim\T^l(V))\cdot\log_p(l)$, then the instability parabolic of any unstable $K$-rational point in $\T^l(V)$ is defined over $K^{p^{-N}}$.
\end{proof}

Let $X$ be a smooth variety over an algebraically closed field $k$ of characteristic $p>0$, $\E$ a locally free sheaf of rank $n$ on $X$. Then the locally free sheaf $\T^l(\E)\subset\E^{\otimes l}$ is defined to be the sheaf of sections of the associated vector bundle of the frame bundle of $\E$ through the representation $\T^l(V)$ (See \cite[Definition 3.4]{Sun08}).

\begin{Theorem}\label{Thm:T(V)}
Let $k$ be an algebraically closed field of characteristic $p>0$, $X$ a smooth projective variety over $k$ with a fixed ample divisor $H$. Let $\E$ be a torsion free sheaf of rank $n$ on $X$ such that $F^{N*}_X(\E)$ is semistable for some integer
$$N\geq\max_{1\leq r\leq N(p,n,l)}C^r_{N(p,n,l)}\cdot\log_p(lr),$$
where $N(p,n,l)=\sum^{l(p)}_{q=0}(-1)^q\cdot C^q_n\cdot C^{l-pq}_{n+l-q-1}$, and $l(p)$ is the unique integer such that $0\leq l-l(p)\cdot p< p$. Then the torsion free sheaf $\T^l(\E)$ is also a semistable sheaf.
\end{Theorem}
\begin{proof}
Let $V$ be a $n$-dimensional $k$-vector space. Then by Proposition \ref{Prop:l}, the $GL_K(V)$-representation on $\T^l(V)$ is a polynomial representation of dimension $N(p,n,l)$ at most degree $l$. Hence this Proposition followed by Theorem \ref{Thm:StbExtGroup}.
\end{proof}

\section{Semistability of Frobenius direct image of $\rho_*(\Omg^1_X)$}

Let $k$ be an algebraically closed field of characteristic $p>0$, $X$ a smooth projective surface with a fixed ample divisor $H$ such that $\Omg^1_X$ is semistable with $\deg_H(\Omg^1_X)>0$, and $\Ls\in\Pic(X)$. Y. Kitadai, H. Sumihiro \cite[Theorem 3.1]{KitadaiSumihiro08} showed that the ${F_X}_*(\Omg^1_X)$ is semistable. Moreover, X. Sun \cite[Theorem 4.11]{Sun10} generalized the \cite[Theorem 3.1]{KitadaiSumihiro08} and showed that ${F_X}_*(\Omg^1_X)$ is stable if $\Omg^1_X$ is stable. In this section, we would like to study the semistability of Frobenius direct image of $\rho_*(\Omg^1_X)$ on higher dimensional base space, where $\rho_*(\Omg^1_X)$ is the locally free sheaf obtained from $\Omg^1_X$ via a $GL_n(k)$-representation $\rho$.

\begin{Lemma}\cite[Theorem 4.8]{Sun08}\label{Lem:SunStb}
Let $k$ be an algebraically closed field of characteristic $p>0$, $X$ a smooth projective variety over $k$ of dimension $n$ with a fixed ample divisor $H$ such that $\deg_H(\Omg^1_X)\geq 0$. Let $\E$ be a torsion free sheaf on $X$ such that $\E\otimes\T^l(\Omg^1_X)$, $1\leq l\leq n(p-1)$, are semistable sheaves. Then the Frobenius direct image ${F_X}_*(\E)$ is a semistable sheaf.
\end{Lemma}

\begin{Lemma}\label{Lem:TensorRep}
Let $K$ be a field, $V$, $W_1$ and $W_2$ are $k$-vector spaces. Let $$\rho_l:GL_K(V)\rightarrow GL_K(W_l)$$ be $GL_K(V)$-representation at most degree $d_l$, $l=1,2$. Then the $GL_K(V)$-representation $$\rho_1\otimes\rho_2:GL_K(V)\rightarrow GL_K(W_1\otimes W_2)$$ at most degree $d_1+d_2$.
\end{Lemma}
\begin{proof}
The $GL_K(V)$-representation $\rho_l$ is given by a matrix $M_l$ of regular functions $$\frac{f^{(l)}_{ij}}{\det(T_{ij})^{a_l}}\in K[T_{ij},\det(T_{ij})^{-1}]_{1\leq i,j\leq\dim_KV},$$ where $a_l\in\N$ and $f^{(l)}_{ij}\in K[T_{ij}]_{1\leq i,j\leq\dim_KV}$ with $\deg(f^{(l)}_{ij})\leq d_l$, $l=1,2$. Then $GL_K(V)$-representation $$\rho_1\otimes\rho_2:GL_K(V)\rightarrow GL_K(W_1\otimes W_2)$$ is determined by the matrix $M_1\otimes M_2$ of regular functions on $GL_K(V)$. Hence it is at most degree $d_1+d_2$.
\end{proof}

\begin{Proposition}\label{Prop:StbTenTru}
Let $k$ be an algebraically closed field of characteristic $p>0$, $X$ a smooth projective variety over $k$ with a fixed ample divisor $H$, and $\rho:GL_n(k)\rightarrow GL_m(k)$ a $GL_n(k)$-representation at most degree $d$. Let $\E$ be a torsion free sheaf of rank $n$ on $X$, and $F^{N*}_X(\E)$ is semistable for some integer
$$N\geq\max_{0<r<m\cdot N(p,n,l)}C^r_{m\cdot N(p,n,l)}\log_p((d+l)r).$$
where $N(p,n,l)=\sum^{l(p)}_{q=0}(-1)^q\cdot C^q_n\cdot C^{l-pq}_{n+l-q-1}$, and $l(p)$ is the unique integer such that $0\leq l-l(p)\cdot p< p$. Then the torsion free part of $\rho_*(\E)\otimes\T^l(\E)$ is also a semistable sheaf, where $\rho_*(\E)$ is the torsion free sheaf obtained from $\E$ via the representation $\rho$.
\end{Proposition}
\begin{proof}
Let $V$ and $W$ be $k$-vector spaces of dimension $n$ and $m$ with fixed bases respectively. Then there are natural isomorphisms $GL_n(k)\cong GL_k(V)$ and $GL_m(k)\cong GL_k(W)$ with respect to the given bases. The corresponding representation $\rho:GL_k(V)\rightarrow GL_k(W)$ is still denoted by $\rho$.

Let the torsion free sheaf $\E$ on $X$ corresponds to the rational $GL_k(V)$-bundle $E$ on $X$. Then the torsion free part of $\rho_*(\E)\otimes\T^l(\E)$ corresponds to the rational $GL_k(W\otimes_k\T^l(V))$-bundle $E(GL_k(W\otimes_k\T^l(V)))$ which is obtained from $E$ by extension of structure group to $GL_k(W\otimes_k\T^l(V))$ via the natural representation $$\widetilde{\rho}^l:=\rho^l\otimes\rho:GL_k(V)\rightarrow GL_k(W\otimes_k\T^l(V)).$$

One observes that $\widetilde{\rho}^l$ is a $GL_k(V)$-representation over $k$ at most degree $d+l$ by Lemma \ref{Lem:TensorRep}. Hence, by Theorem \ref{Thm:StbExtGroup}, if $F^{N*}_X(\E)$ is semistable for some integer
$$N\geq\max_{0<r<m\cdot N(p,n,l)}C^r_{m\cdot N(p,n,l)}\log_p((d+l)r),$$
then $E(GL_k(W\otimes_k\T^l(V)))$ is a semistable rational $GL_k(W\otimes_k\T^l(V))$-bundle, i.e., the torsion free part of $\rho_*(\E)\otimes\T^l(\E)$ is a semistable sheaf.
\end{proof}

\begin{Theorem}\label{Thm:FroRep}
Let $k$ be an algebraically closed field of characteristic $p>0$, $X$ a smooth projective variety over $k$ of dimension $n$ with a fixed ample divisor $H$ such that $\deg_H(\Omg^1_X)\geq 0$, $\rho:GL_n(k)\rightarrow GL_m(k)$ a $GL_n(k)$-representation at most degree $d$. Suppose that $F^{N*}_X(\Omg^1_X)$ is semistable for some integer
$$N\geq\max_{1\leq l\leq n(p-1)\atop 0<r<m\cdot N(p,n,l)}C^r_{m\cdot N(p,n,l)}\log_p((d+l)r).$$
where $N(p,n,l)=\sum^{l(p)}_{q=0}(-1)^q\cdot C^q_n\cdot C^{l-pq}_{n+l-q-1}$, and $l(p)$ is the unique integer such that $0\leq l-l(p)\cdot p< p$. Then the Frobenius direct image ${F_X}_*(\rho_*(\Omg^1_X))$ is a semistable sheaf, where $\rho_*(\Omg^1_X)$ is the locally free sheaf obtained from $\Omg^1_X$ via the representation $\rho$.
\end{Theorem}
\begin{proof}
By Proposition \ref{Prop:StbTenTru}, we have $\T^l(\Omg^1_X)\otimes\rho_*(\Omg^1_X)$ are semistable sheaves for all integers $1\leq l\leq n(p-1)$. Then the Frobenius direct image ${F_X}_*(\rho_*(\Omg^1_X))$ is a semistable sheaf by Lemma \ref{Lem:SunStb}.
\end{proof}

\begin{Corollary}
Let $k$ be an algebraically closed field of characteristic $p>0$, $X$ a smooth projective variety over $k$ of dimension $n$ with a fixed ample divisor $H$ such that $\deg_H(\Omg^1_X)\geq 0$, $d$ and $m\in\N_+$. Suppose that $F^{N*}_X(\Omg^1_X)$ is semistable for some integer $$N\geq\max_{1\leq l\leq n(p-1)\atop 0<r<m\cdot N(p,n,l)}C^r_{m\cdot N(p,n,l)}\log_p((d+l)r).$$
where $N(p,n,l)=\sum^{l(p)}_{q=0}(-1)^q\cdot C^q_n\cdot C^{l-pq}_{n+l-q-1}$, and $l(p)$ is the unique integer such that $0\leq l-l(p)\cdot p< p$. Then
\begin{itemize}
    \item[$(1)$] If $m=n^d$, then ${F_X}_*((\Omg^1_X)^{\otimes d})$ is a semistable sheaf.
    \item[$(2)$] If $m=C^{r-1}_{n+r-1}$, then ${F_X}_*(\Sym^d(\Omg^1_X))$ is a semistable sheaf.
    \item[$(3)$] If $m=C^d_n$, then ${F_X}_*(\bigwedge^d(\Omg^1_X))$ is a semistable sheaf.
\end{itemize}
\end{Corollary}
\begin{proof}
It is obvious that for any integer $d>0$, the natural $GL_k(V)$-action on $V^{\otimes d}$, $\Sym^d(V)$ and $\bigwedge^d(V)$ are all polynomial representation of degree $d$. Since $\rk((\Omg^1_X)^{\otimes d})=n^d$, $\rk(\Sym^d(\Omg^1_X))=C^{r-1}_{n+r-1}$ and $\rk(\bigwedge^d(\Omg^1_X))=C^d_n$, this corollary follows from Theorem \ref{Thm:FroRep}.
\end{proof}

\noindent\textbf{Acknowledgments:}
I would like to express my hearty thanks to Professor Xiaotao Sun, who introduced me to this subject. This work was supported by National Natural Science Foundation of China (Grant No. 11501418), Shanghai Sailing Program(15YF1412500).

\end{document}